\documentclass{amsart}
\usepackage{amssymb, amsmath, amsthm, graphics, comment, xspace, enumerate}
\newtheorem{thm}{Theorem}[section]

\newtheorem{prop}[thm]{Proposition}
\theoremstyle{definition}
\newtheorem{defn}[thm]{Definition}
\newtheorem{example}[thm]{Example}
\theoremstyle{remark}
\newtheorem{rem}[thm]{Remark}
\numberwithin{equation}{section}

\newcommand{\DD}{\mathcal D}

\begin{document}
\title[Frequently hypercyclic $C$-dist...]{Frequently hypercyclic $C$-distribution semigroups and their generalizations}

\author{Marko Kosti\' c}
\address{Faculty of Technical Sciences,
University of Novi Sad,
Trg D. Obradovi\' ca 6, 21125 Novi Sad, Serbia}
\email{marco.s@verat.net}

{\renewcommand{\thefootnote}{} \footnote{
\\ \text{  }  \ \    {\it Key words and phrases.} $C$-distribution semigroups, integrated $C$-semigroups, $f$-frequent hypercyclicity, $q$-frequent hypercyclicity, Fr\' echet spaces.
\\  \text{  }  \ \ This research is partially supported by grant 174024 of Ministry
of Science and Technological Development, Republic of Serbia.}}

\begin{abstract}
In this paper, we introduce the notions of $f$-frequent hypercyclicity and ${\mathcal F}$-hypercyclicity for $C$-distribution semigroups in separable Fr\'echet spaces. We particularly analyze the classes of $q$-frequently hypercyclic $C$-distribution semigroups ($q\geq 1$) and frequently hypercyclic $C$-distribution semigroups, providing a great number of illustrative examples. 
\end{abstract}
\maketitle

\section{Introduction and Preliminaries}

The notion of a frequently hypercyclic linear continuous operator on a separable Fr\'echet space was introduced by F. Bayart and S. Grivaux in 2006 (\cite{bay1}). The general notion of $(m_k)$-hypercyclicity for linear continuous operators was introduced by F. Bayart and \'E. Matheron \cite{Baya1} in 2009, while some special cases of $(m_k)$-hypercyclicity, like $q$-frequent hypercyclicity ($q\in {\mathbb N}$), were analyzed by M. Gupta and A. Mundayadan in \cite{gupta}. 
Within the field of linear topological dynamics, the notion of ${\mathcal F}$-hypercyclicity, where ${\mathcal F}$ is a Furstenberg family, was introduced for the first time by S. Shkarin in 2009 (\cite{shkarin}); further contributions were given by J. B\`es, Q. Menet, A. Peris, Y. Puig \cite{biba-prim} and A. Bonilla,  K.-G. Grosse-Erdmann \cite{boni-upper}.
The notion of ${\mathcal F}$-hypercyclicity for linear not necessarily continuous operators has been recently introduced by the author in \cite{1211212018}. For more details on the subject, we refer the reader to \cite{Bay}-\cite{biba}, \cite{Grosse},  \cite{menet}-\cite{measures} and references cited therein.

On the other hand, the notion of a frequently hypercyclic strongly continuous semigroup on a separable Banach space was introduced by E. M. Mangino and A. Peris in 2011 (\cite{man-peris}). Frequently hypercyclic translation semigroups on weighted function spaces were further investigated by  E. M. Mangino and M. Murillo-Arcila in \cite{man-marina}, while the frequent hypercyclity of semigroup solutions for first-order partial differential equations arising in mathematical biology was investigated by C.-H. Hung and Y.-H. Chang in \cite{hung}. Frequent hypercyclicity and various generalizations of this concept for single operators and semigroups of operators are still very active field of research, full of open unsolved problems.

Hypercyclicity of $C$-regularized semigroups, distribution semigroups and unbounded linear operators in Banach spaces was analyzed by R. deLaubenfels, H. Emamirad and K.-G. Grosse-Erdmann in 2003 (\cite{cycch}). The non-existence of an appropriate reference which treats the frequent hypercyclicity of $C$-regularized semigroups and distribution semigroups strongly influenced us to write this paper. 

We work in the setting of separable infinite-dimensional Fr\'echet spaces, considering general classes of $C$-distribution semigroups and fractionally integrated $C$-semigroups (\cite{knjigah}-\cite{knjigaho}); here, we would like to point out that our results seem to be new even for strongly continuous semigroups of operators in Fr\'echet spaces. In contrast to the investigations of frequently hypercyclic strongly continuous semigroups of operators that are carried out so far, the notion of Pettis integrability does not play any significant role in our approach, which is primarly oriented for giving some new applications in the qualitative analysis of solutions of abstract ill-posed differential equations of first order. The notion of a $q$-frequently hypercyclic strongly continuous semigroup, where $q\geq 1,$ has been recently introduced and systematically analyzed in our joint paper with B. Chaouchi, S. Pilipovi\' c and D. Velinov \cite{qjua}; the notion of $f$-frequent hypercyclicity, introduced here for the first time as a continuous counterpart of $(m_k)$-hypercyclicity, seems to be not considered elsewhere even for strongly continuous semigroups of operators in Banach spaces. Albeit we analyze the general class of $C$-distribution semigroups, providing also some examples of frequently hypercyclic integrated semigroups, almost all structural results of ours are stated for the class of global $C$-regularized semigroups (for certain difficuties we have met in our exploration of frequently hypercyclic fractionally integrated $C$-semigroups, we refer the reader to Remark \ref{prcko-tres}). Without any doubt, our main theoretical result is Theorem \ref{oma}, which can be called $f$-Frequent Hypercyclicity Criterion for $C$-Regularized Semigroups. In Theorem \ref{oma-duo}, we state
Upper Frequent Hypercyclicity Criterion for $C$-Regularized Semigroups (this result seems to be new even for strongly continuous semigroups in Banach spaces, as well). From the point of view of possible applications, Theorem \ref{3.1.4.13}, in which we reconsider the spectral criterions esatablished by S. El Mourchid \cite[Theorem 2.1]{samir} and E. M. Mangino, A. Peris \cite[Corollary 2.3]{man-peris}, and Theorem \ref{3018}, in which we reconsider the famous Desch-Schappacher-Webb criterion for chaos of  strongly continuous semigroups \cite[Theorem 3.1]{fund}, are most important; both theorems are consequences of Theorem \ref{oma}. In Example \ref{freja}, we revisit \cite[Subsection 3.1.4]{knjigah} and prove that all examined $C$-regularized semigroups and integrated semigroups, including corresponding single linear operators, are frequently hypercyclic (we already know that these semigroups and operators are topologically mixing or chaotic in a certain sense). 

We use the standard notation throughout the paper. By $E$ we denote a separable infinite-dimensional Fr\' echet space (real or complex). We assume that
the topology of $E$ is induced by the fundamental system
$(p_{n})_{n\in {\mathbb N}}$ of increasing seminorms. If $Y$ is also a Fr\' echet space, over the same field of scalars ${\mathbb K}$ as $E,$ then by $L(E,Y)$ we denote the space consisting of all continuous linear mappings from $E$ into $Y.$ The
translation invariant metric $d: E\times E \rightarrow [0,\infty),$ defined by
$$
d(x,y):=\sum
\limits_{n=1}^{\infty}\frac{1}{2^{n}}\frac{p_{n}(x-y)}{1+p_{n}(x-y)},\
x,\ y\in E,
$$
satisfies, among many other properties, the following ones:
$d(x+u,y+v)\leq d(x,y)+d(u,v)$ and $d(cx,cy)\leq (|c|+1)d(x,y),\
c\in {\mathbb K},\ x,\ y,\ u,\ v\in X.$ Set $L(x,\epsilon):=\{y\in X : d(x,y)<\epsilon\}$ and $L_{n}(x,\epsilon):=\{y\in X : p_{n}(x-y)<\epsilon\}$ ($n\in {\mathbb N},$ $\epsilon>0,$ $x\in X$). By $E^{\ast}$ we denote the dual space of $E.$ For a closed linear operator $T$ on $E,$ we denote by $D(T),$ $R(T),$ $N(T),$ $\rho(T)$ and $\sigma_{p}(T)$ its domain, range, kernel, resolvent set and point spectrum, respectively. If ${\tilde E}$ is a linear subspace of $E,$ then the part of $T$ in ${\tilde E},$ $T_{|{\tilde E}}$ shortly, is defined through $T_{|{\tilde E}}:=\{(x,y) \in T : x,\ y\in {\tilde E}\}$ (we will identify an operator and its graph henceforth).
Set $D_{\infty}(T):=\bigcap_{k\in {\mathbb N}}D(T^{k}).$ 
We will always assume henceforth that $C\in L(E)$ and $C$ is injective. Put $p_{C}(x):=p(C^{-1}x),$ $p\in \circledast,$ $x\in R(C).$ Then
$p_{C}(\cdot)$ is a seminorm on $R(C)$ and the calibration
$(p_{C})_{p\in \circledast}$ induces a Fr\'echet topology on
$R(C);$ we denote this space by $[R(C)]_{\circledast}.$  
If $T^{k}$ is closed for any $k\in {\mathbb N},$ then the space
$C(D(T^{k})),$ equipped with the following family of seminorms $p_{k,n}(Cx):=p_{n}(x)+p_{n}(Tx)+\cdot \cdot \cdot +p_{n}(T^{k}x),$ $x\in D(T^{k}),$ is a Fr\'echet one ($n\in {\mathbb N}$). This space will be denoted by $[C(D(T^{k}))].$
For any $s\in {\mathbb R},$ we define $\lfloor s \rfloor :=\sup \{
l\in {\mathbb Z} : s\geq l \}$ and $\lceil s \rceil :=\inf \{ l\in
{\mathbb Z} : s\leq l \}.$ 

Let us recall that a series $\sum_{n=1}^{\infty}x_{n}$ in $E$ is called unconditionally convergent iff for every permutation $\sigma$ of ${\mathbb N}$,
the series $\sum_{n=1}^{\infty}x_{\sigma (n)}$  is convergent; it is well known that the absolute convergence of $\sum_{n=1}^{\infty}x_{n}$ (i.e., the convergence of $\sum_{n=1}^{\infty}p_{l}(x_{n})$ for all $l\in {\mathbb N}$) implies its unconditional convergence (see \cite{boni} and references cited therein for further information on the subject).

The Schwartz space of rapidly decreasing functions $\mathcal{S}$ is defined by the following system of seminorms
$
p_{m,n}(\psi):=\sup_{x\in\mathbb{R}}\, |x^m\psi^{(n)}(x)|,$
$\psi\in\mathcal{S},$ $m,\ n\in\mathbb{N}_0.$
We use notation  $\mathcal{D}=C_0^{\infty}(\mathbb{R})$
and $\mathcal{E}=C^{\infty}(\mathbb{R})$.
If $\emptyset \neq \Omega  \subseteq {\mathbb R},$ then the symbol $\mathcal{D}_{\Omega}$ denotes the subspace of $\mathcal{D}$ consisting of those functions $\varphi \in \mathcal{D}$ for which supp$(\varphi) \subseteq \Omega;$ $\mathcal{D}_{0}\equiv \mathcal{D}_{[0,\infty)}.$
The spaces
$\mathcal{D}'(E):=L(\mathcal{D},E)$,
$\mathcal{E}'(E):=L(\mathcal{E},E)$ and
$\mathcal{S}'(E):=L(\mathcal{S},E)$
are topologized in the usual way; the symbols
$\mathcal{D}'_{\Omega}(E)$,
$\mathcal{E}'_{\Omega}(E)$ and $\mathcal{S}'_{\Omega}(E)$ denote their subspaces containing $E$-valued distributions
whose supports are contained in $\Omega ;$ $\mathcal{D}'_{0}(E)\equiv \mathcal{D}'_{[0,\infty)}(E)$, $\mathcal{E}'_{0}(E)\equiv \mathcal{E}'_{[0,\infty)}(E)$, $\mathcal{S}'_{0}(E)\equiv \mathcal{S}'_{[0,\infty)}(E).$ By $\delta_{t}$ we denote the Dirac distribution centered at point $t$ ($t\in {\mathbb R}$).
If $\varphi$, $\psi:\mathbb{R}\to\mathbb{C}$ are
measurable functions, then we define 
$
\varphi*_0
\psi(t):=\int^t_0\varphi(t-s)\psi(s)\,ds,\;t\;\in\mathbb{R}.
$
The convolution of vector-valued distributions will be taken in the sense of \cite[Proposition 1.1]{ku112}.

\subsection{$C$-distribution semigroups and fractionally integrated $C$-semigroups}\label{peru}

Let $C\in L(E)$ be an injective operator, and let $\mathcal{G}\in\mathcal{D}_0'^{\ast}(L(E))$ satisfy $C\mathcal{G}=\mathcal{G}C$.
Then we say that $\mathcal{G}$ is a $C$-distribution semigroup, shortly (C-DS), iff ${\mathcal G}$ satisfies the following two conditions:
\begin{itemize}
\item[(i)] ${\mathcal G}(\varphi\ast_0\psi)C={\mathcal G}(\varphi){\mathcal G}(\psi)$, $\varphi,\, \psi\in\DD$;\\
\item[(ii)] ${\mathcal N}({\mathcal G}):=\bigcap\limits_{\varphi\in\DD_0}N({\mathcal G}(\varphi))=\{0\}$.
\end{itemize}
If, additionally, ${\mathcal{R}}(\mathcal{G}):=\bigcup_{\varphi\in\mathcal{D}_0}R(\mathcal{G}(\varphi))$
is dense in $E$, then we say that 
${\mathcal G}$ is a dense $C$-distribution semigroup.

Let $T\in\mathcal{E}_0',$ 
i.e., $T$ is a scalar-valued distribution with compact support contained in $[0,\infty)$.
Set
\[
G(T)x:=\bigl\{(x,y) \in E\times E : \mathcal{G}(T*\varphi)x=\mathcal{G}(\varphi)y\mbox{ for all }\varphi\in\mathcal{D}_0  \bigr\}.
\]
Then  it can be easily seen that $G(T)$ is a closed linear operator.
We define the (infinitesimal) generator of a (C-DS) $\mathcal{G}$ by $A:=G(-\delta').$ 

Suppose that $\mathcal{G}$ is a (C-DS). Then
$\mathcal{G}(\varphi)\mathcal{G}(\psi)=\mathcal{G}(\psi)\mathcal{G}(\varphi)$ for all $\varphi,\,\psi\in\mathcal{D},$ $C^{-1}AC=A$ and the following holds:
Let $S$, $T\in\mathcal{E}'_0$, $\varphi\in\mathcal{D}_0$,  $\psi\in\mathcal{D}$ and $x\in E$.
Then we have:
\begin{itemize}
\item[A1.] $G(S)G(T)\subseteq G(S*T)$ with $D(G(S)G(T))=D(G(S*T))\cap D(G(T))$, and $G(S)+G(T)\subseteq G(S+T)$.
\item[A2.] $(\mathcal{G}(\psi)x$, $\mathcal{G}(-\psi^{\prime})x-\psi(0)Cx)\in A$.
\end{itemize}

We denote by $D({\mathcal G})$ the set consisting of those elements $x\in E$ for which $x\in D(G({\delta}_t)),$ $t\geq 0$ and the mapping $t\mapsto G({\delta}_t)x,$ $t\geq 0$ is continuous. By A1., we have 
that
\begin{align*}
D\bigl(G(\delta_s)G(\delta_t)\bigr)\!=\!D\bigl(G(\delta_s*\delta_t)\bigr)\cap D\bigl(G(\delta_t)\bigr)\!=\!D\bigl(G(\delta_{t+s})\bigr)\cap D\bigl(G(\delta_t)\bigr),
\;t,\,s\!\geq\! 0,
\end{align*}
which clearly implies $G(\delta_t)(D(\mathcal{G}))\subseteq D(\mathcal{G})$, $t\geq 0$ and 
\begin{equation}\label{C-DS}
G\bigl(\delta_s\bigr)G\bigl(\delta_t
\bigr)x=G\bigl(\delta_{t+s}\bigr)x, \quad t,\,s\geq 0,\ x\in D(\mathcal{G}).
\end{equation}

The following definition is well-known:

\begin{defn}\label{first}
Let $\alpha \geq 0,$ and let $A$ be a closed linear operator. If
there exists a strongly continuous operator family
$(S_\alpha(t))_{t\geq 0}\subseteq L(E)$ such that:
\begin{itemize}
\item[(i)] $S_\alpha(t)A\subseteq AS_\alpha(t)$, $t\geq 0$,
\item[(ii)] $S_\alpha(t)C=CS_\alpha(t)$, $t\geq 0$,
\item[(iii)] for all $x\in E$ and $t\geq 0$: $\int_0^tS_\alpha(s)x\,ds\in D(A)$ and
\begin{align*}
A\int\limits_0^tS_\alpha(s)x\,ds=S_\alpha(t)x-g_{\alpha +1}(t)Cx,
\end{align*}
\end{itemize}
then it is said that $A$ is a subgenerator of a (global)
$\alpha$-times integrated $C$-semigroup $(S_\alpha(t))_{t\geq 0}$.
Furthermore, it is said that $(S_\alpha(t))_{t\geq 0}$ is an
exponentially equicontinuous, $\alpha$-times integrated
$C$-semigroup with a subgenerator $A$ if, in addition, there exists
$\omega \in {\mathbb R}$ such that the family $\{e^{-\omega
t}S_{\alpha}(t) : t\geq 0\}\subseteq L(E)$ is equicontinuous.
\end{defn}

If $\alpha =0,$ then $(S_0(t))_{t\geq 0}$ is
also said to be a $C$-regularized semigroup with subgenerator $A;$ in this case, we have the following simple functional equation: 
$S_0(t)S_0(s)=S_0(t+s)C,$ $t,\ s\geq 0.$
Moreover, if $\alpha \geq 0$ and
$(S_\alpha(t))_{t\geq 0}$ is a
global $\alpha$-times integrated $C$-semigroup with subgenerator $A,$ then
$(S_\alpha(t))_{t\geq 0}$ is
locally equicontinuous, i.e., the following holds: for every $T>0$
and $n\in {\mathbb N},$ there exist $m\in {\mathbb N}$ and $c>0$ such
that $p_{n}(S_{\alpha}(t)x)\leq cp_{m}(x),$ $t\in [0,T],$ $x\in E.$
The integral generator of $(S_\alpha(t))_{t\geq 0}$ is defined by
\begin{align*}
\hat{A}:=\Biggl\{(x,y)\in E\times
E:S_\alpha(t)x-g_{\alpha+1}(t)Cx=\int\limits^t_0S_\alpha(s)y\,ds,\;t\geq
0\Biggr\}.
\end{align*}
The integral generator of $(S_\alpha(t))_{t\geq 0}$
is a closed linear operator which is an
extension of any subgenerator of $(S_\alpha(t))_{t\geq 0}.$
Arguing as in the proofs of
\cite[Proposition 2.1.6, Proposition 2.1.19]{knjigah}, we may deduce that
the integral generator of $(S_\alpha(t))_{t\geq 0}$
is its maximal subgenerator with respect
to the set inclusion. Furthermore, the
following equality holds $\hat{A}=C^{-1}AC.$

Let $A$ be a closed linear operator on $E.$ Denote by $Z_{1}(A)$ the space consisting of
those elements $x\in E$ for which there exists a unique
continuous mapping $u :[0,\infty) \rightarrow E$ satisfying $\int^t_0 u(s,x)\,ds\in D(A)$ and
$A\int^t_0 u(s,x)\,ds=u(t,x)-x$, $t\geq 0,$ i.e., the unique mild
solution of the corresponding Cauchy problem $(ACP_{1}):$
$$
(ACP_{1}) : u^{\prime}(t)=Au(t),\ t\geq 0, \ u(0)=x.
$$

Suppose now that $A$ is a subgenerator of a
global $\alpha$-times integrated $C$-semigroup $(S_\alpha(t))_{t\geq
0}$ for some $\alpha \geq 0.$ Then there is only one (trivial) mild solution of $
(ACP_{1})$ with $x=0,$ so that $Z_{1}(A)$ is a linear subspace of $X.$
Moreover, for every $\beta>\alpha ,$ the
operator $A$ is a subgenerator (the integral generator) of a global
$\beta$-times integrated $C$-semigroup $(S_\beta(t)\equiv
(g_{\beta-\alpha}\ast S_{\alpha}\cdot)(t))_{t\geq 0},$ where $g_{\zeta}(t):=t^{\zeta-1}/\Gamma(\zeta)$ for $t>0$ and $\Gamma (\cdot)$ denotes the Gamma function ($\zeta>0$). The
space
$Z_{1}(A)$
consists exactly of those elements $x\in E$ for which the mapping
$t\mapsto C^{-1}S_{\lceil \alpha \rceil}(t)x,$ $t\geq 0$ is well
defined and $\lceil \alpha \rceil$-times continuously differentiable
on $[0,\infty);$ see e.g. \cite{knjigaho}. Set
\begin{align}\label{monman}
{\mathcal G}(\varphi)x:=(-1)^{\lceil \alpha \rceil}\int
\limits^{\infty}_{0}\varphi^{(\lceil \alpha \rceil)}(t)S_{\lceil
\alpha \rceil}(t)x\, dt,\quad \varphi \in {\mathcal D}_{{\mathbb
K}},\ x\in E
\end{align}
and
$$
G\bigl(\delta_{t}\bigr)x:=\frac{d^{\lceil \alpha \rceil}}{dt^{\lceil
\alpha \rceil}}C^{-1}S_{\lceil \alpha \rceil}(t)x, \quad t\geq 0,\
x\in Z_{1}(A).
$$
Then ${\mathcal G}$ is a $C$-distribution semigroup generated by $C^{-1}AC$ and $Z_{1}(A)=D({\mathcal G})$
(see e.g.
\cite{knjigah} and \cite{C-ultra}). 
Before proceeding further, it should be observed that the solution
space $Z_{1}(A)$ is independent of the choice of
$(S_\alpha(t))_{t\geq 0}$ in the
following sense: If $C_{1}\in L(X)$ is another injective operator
with $C_{1}A\subseteq AC_{1},$ $\gamma \geq 0,$ $x\in E$ and $A$ is
a subgenerator (the integral generator) of a global $\gamma$-times
integrated $C_{1}$-semigroup $(S^\gamma(t))_{t\geq 0},$ then the mapping $t\mapsto
C^{-1}S_{\lceil \alpha \rceil}(t)x,$ $t\geq 0$ is well defined and
$\lceil \alpha \rceil$-times continuously differentiable on
$[0,\infty)$ iff the mapping $t\mapsto C_{1}^{-1}S^{\lceil \gamma
\rceil}(t)x,$ $t\geq 0$ is well defined and $\lceil \gamma
\rceil$-times continuously differentiable on $[0,\infty)$ (with the
clear notation). If this is the case, then we have that
$u(t;x):=G(\delta_{t})x=\frac{d^{\lceil \gamma \rceil}}{dt^{\lceil
\gamma \rceil}}C_{1}^{-1}S^{\lceil \gamma \rceil}(t)x,$ $t\geq 0$
is a unique mild
solution of the corresponding Cauchy problem $(ACP_{1}).$ Furthermore,
$(S_\alpha(t))_{t\geq 0}$ and $(S^\gamma(t))_{t\geq
0}$
share the same (subspace) $f$-frequently hypercyclic properties defined below.

We refer the reader to \cite{knjigah}-\cite{C-ultra} for further information concerning $C$-distribution semigroups. The notion of exponentially equicontinuous, analytic fractionally integrated $C$-semigroups will be taken 
in a broad sense of \cite[Definition 2.2.1(i)]{knjigaho}, while the notion of an entire $C$-regularized group will be taken 
in the sense of \cite[Definition 2.2.9]{knjigaho}. For more details about $C$-regularized semigroups and their applications, we refer the reader to the monograph \cite{l1} by R. deLaubenfels.

\subsection{Lower and upper densities}
First of all, we need to recall the following definitions from \cite{1211212018}:

\begin{defn}\label{4-skins-MLO-okay}
Let $(T_{n})_{n\in {\mathbb N}}$ be a sequence of linear operators acting between the spaces
$X$ and $Y,$ let $T$ be a linear operator on $X$, and let $x\in X$. Suppose that ${\mathcal F}\in P(P({\mathbb N}))$ and ${\mathcal F}\neq \emptyset.$ Then we say
that:
\begin{itemize}
\item[(i)] $x$ is an ${\mathcal F}$-hypercyclic element of the sequence
$(T_{n})_{n\in {\mathbb N}}$ iff
$x\in \bigcap_{n\in {\mathbb N}} D(T_{n})$ and for each open non-empty subset $V$ of $Y$ we have that 
$$
S(x,V):=\bigl\{ n\in {\mathbb N} : T_{n}x \in V \bigr\}\in {\mathcal F} ;
$$ 
$(T_{n})_{n\in {\mathbb N}}$ is said to be ${\mathcal F}$-hypercyclic iff there exists an ${\mathcal F}$-hypercyclic element of
$(T_{n})_{n\in {\mathbb N}}$;
\item[(ii)] $T$ is ${\mathcal F}$-hypercyclic iff the sequence
$(T^{n})_{n\in {\mathbb N}}$ is ${\mathcal F}$-hypercyclic; $x\in D_{\infty}(T)$ is said to be
an ${\mathcal F}$-hypercyclic element of $T$ iff $x$ is an ${\mathcal F}$-hypercyclic element of the sequence
$(T^{n})_{n\in {\mathbb N}}.$
\end{itemize}
\end{defn}

\begin{defn}\label{prckojed}
Let $q\in [1,\infty),$ let $A \subseteq {\mathbb N}$, and let $(m_{n})$ be an increasing sequence in $[1,\infty).$ Then:
\begin{itemize}
\item[(i)] The lower $q$-density of $A,$ denoted by $\underline{d}_{q}(A),$ is defined through:
$$
\underline{d}_{q}(A):=\liminf_{n\rightarrow \infty}\frac{|A \cap [1,n^{q}]|}{n}.
$$
\item[(ii)] The upper $q$-density of $A,$ denoted by $\overline{d}_{q}(A),$ is defined through:
$$
\overline{d}_{q}(A):=\limsup_{n\rightarrow \infty}\frac{|A \cap [1,n^{q}]|}{n}.
$$
\item[(iii)] The lower $(m_{n})$-density of $A,$ denoted by $\underline{d}_{m_{n}}(A),$ is defined through:
$$
\underline{d}_{{m_{n}}} (A):=\liminf_{n\rightarrow \infty}\frac{|A \cap [1,m_{n}]|}{n}.
$$
\item[(iv)] The upper $(m_{n})$-density of $A,$ denoted by $\overline{d}_{{m_{n}}}(A),$ is defined through:
$$
\overline{d}_{{m_{n}}}(A):=\limsup_{n\rightarrow \infty}\frac{|A \cap [1,m_{n}]|}{n}.
$$
\end{itemize}
\end{defn}

Assume that $q\in [1,\infty)$ and $(m_{n})$ is an increasing sequence  in $[1,\infty).$ Consider the notion introduced in Definition \ref{4-skins-MLO-okay} with:
(i) ${\mathcal F}=\{A \subseteq {\mathbb N} : \underline{d}(A)>0\},$ (ii) ${\mathcal F}=\{A \subseteq {\mathbb N} : \underline{d}_{q}(A)>0\},$
(iii) ${\mathcal F}=\{A \subseteq {\mathbb N} : \underline{d}_{{m_{n}}}(A)>0\};$ 
then we say that $(T_{n})_{n\in {\mathbb N}}$ ($T,$ $x$) is frequently hypercyclic, $q$-frequently hypercyclic and
l-$(m_{n})$-hypercyclic, respectively. 

Denote by $m(\cdot)$ the Lebesgue measure on $[0,\infty).$ 
We would like to propose the following definition:

\begin{defn}\label{prckojed-prim}
Let $q\in [1,\infty),$ let $A\subseteq [0,\infty)$, and let $f : [0,\infty) \rightarrow [1,\infty)$ be an increasing mapping. Then: 
\begin{itemize}
\item[(i)] The lower $qc$-density of $A,$ denoted by $\underline{d}_{qc}(A),$ is defined through:
$$
\underline{d}_{qc}(A):=\liminf_{t\rightarrow \infty}\frac{m(A \cap [0,t^{q}])}{t}.
$$
\item[(ii)] The upper $qc$-density of $A,$ denoted by $\overline{d}_{qc}(A),$ is defined through:
$$
\overline{d}_{qc}(A):=\limsup_{t\rightarrow \infty}\frac{m(A \cap [0,t^{q}])}{t}.
$$
\item[(iii)] The lower $f$-density of $A,$ denoted by $\underline{d}_{f}(A),$ is defined through:
$$
\underline{d}_{f} (A):=\liminf_{t\rightarrow \infty}\frac{m(A \cap [0,f(t)])}{t}.
$$
\item[(iv)] The upper $f$-density of $A,$ denoted by $\overline{d}_{f}(A),$ is defined through:
$$
\overline{d}_{f}(A):=\limsup_{t\rightarrow \infty}\frac{m(A \cap [0,f(t)])}{t}.
$$
\end{itemize}
\end{defn}

It is clear that Definition \ref{prckojed-prim} provides continues analogues of the notion introduced in Definition \ref{prckojed}, which have been analyzed in \cite[Section 2]{1211212018} in more detail. For the sake of brevity and better exposition, we will skip all related details about possibilities to transfer the results established in \cite{1211212018} for continuous lower and upper densities.

\section[Generalized frequent hypercyclicity for $C$-distribution semigroups...]{Generalized frequent hypercyclicity for $C$-distribution semigroups and fractionally integrated $C$-semigroups}\label{srboljub}

Let $P([0,\infty))$ denote the power set of $[0,\infty).$ We would  like to propose the following general definition:

\begin{defn}\label{4-skins}
Let ${\mathcal G}$ be a $C$-distribution semigroup, and let $x\in D({\mathcal G})$. Suppose that ${\mathcal F}\in P(P([0,\infty)))$ and ${\mathcal F}\neq \emptyset.$ Then we say
that $x$ is an ${\mathcal F}$-hypercyclic element of ${\mathcal G}$ iff
for each open non-empty subset $V$ of $E$ we have 
$$
S(x,V):=\bigl\{ t\geq 0 : G(\delta_t)x \in V \bigr\}\in {\mathcal F} ;
$$ 
${\mathcal G}$ is said to be ${\mathcal F}$-hypercyclic iff there exists an ${\mathcal F}$-hypercyclic element of
${\mathcal G}.$
\end{defn}

The notion introduced in the following definition is a special case of the notion introduced above, with ${\mathcal F}$ being the collection of all non-empty subsets $A$ of $[0,\infty)$ such that the lower $qc$-density of $A,$ the upper $qc$-density of $A,$ the lower $f$-density of $A$ or the upper $f$-density of $A$ is positive:

\begin{defn}\label{prckojedd}
Let $q\in [1,\infty),$ and let $f : [0,\infty) \rightarrow [1,\infty)$ be an increasing mapping. Suppose that ${\mathcal G}$ is a $C$-distribution semigroup. Then we say that:
\begin{itemize}
\item[(i)] ${\mathcal G}$ is $q$-frequently hypercyclic iff there exists $x\in D({\mathcal G})$ such that for each open non-empty subset $V$ of $E$ we have $\underline{d}_{qc}(\{ t\geq 0 : G(\delta_t)x \in V \bigr\})>0;$
\item[(ii)] ${\mathcal G}$ is upper $q$-frequently hypercyclic iff there exists $x\in D({\mathcal G})$ such that for each open non-empty subset $V$ of $E$ we have $\overline{d}_{qc}(\{ t\geq 0 : G(\delta_t)x \in V \bigr\})>0;$
\item[(iii)] ${\mathcal G}$ is $f$-frequently hypercyclic iff there exists $x\in D({\mathcal G})$ such that for each open non-empty subset $V$ of $E$ we have $\underline{d}_{f}(\{ t\geq 0 : G(\delta_t)x \in V \bigr\})>0;$
\item[(iv)] ${\mathcal G}$ is upper $f$-frequently hypercyclic iff there exists $x\in D({\mathcal G})$ such that for each open non-empty subset $V$ of $E$ we have $\underline{d}_{f}(\{ t\geq 0 : G(\delta_t)x \in V \bigr\})>0.$
\end{itemize}
\end{defn}

It seems natural to reformulate the notion introduced in the previous two definitions for fractionally integrated $C$-semigroups:

\begin{defn}\label{prcko-raki}
Suppose that $A$ is a subgenerator of a
global $\alpha$-times integrated $C$-semigroup $(S_\alpha(t))_{t\geq
0}$ for some $\alpha \geq 0.$ Let ${\mathcal F}\in P(P([0,\infty)))$ and ${\mathcal F}\neq \emptyset.$ Then we say that an element $x\in Z_{1}(A)$ is an ${\mathcal F}$-hypercyclic element of $(S_\alpha(t))_{t\geq
0}$ iff $x$ is an ${\mathcal F}$-hypercyclic element of the induced $C$-distribution semigroup ${\mathcal G}$ defined through \eqref{monman}; 
$(S_\alpha(t))_{t\geq
0}$ is said to be ${\mathcal F}$-hypercyclic iff ${\mathcal G}$ is said to be ${\mathcal F}$-hypercyclic.
\end{defn}

\begin{defn}\label{prcko-raki}
Suppose that $A$ is a subgenerator of a
global $\alpha$-times integrated $C$-semigroup $(S_\alpha(t))_{t\geq
0}$  for some $\alpha \geq 0.$ Let $q\in [1,\infty),$ and let $f : [0,\infty) \rightarrow [1,\infty)$ be an increasing mapping. Then it is said that $(S_\alpha(t))_{t\geq
0}$ is $q$-frequently hypercyclic (upper $q$-frequently hypercyclic, $f$-frequently hypercyclic, upper $f$-frequently hypercyclic) iff the induced $C$-distribution semigroup ${\mathcal G}$, defined through \eqref{monman}, is.
\end{defn}

As mentioned in the introductory part, the following result can be viewed as $f$-Frequent Hypercyclicity Criterion for $C$-Regularized Semigroups:

\begin{thm}\label{oma}
Suppose that $A$ is a subgenerator of a global $C$-regularized semigroup $(S_{0}(t))_{t\geq 0}$ on $E$ and $f : [0,\infty) \rightarrow [1,\infty)$ is an increasing mapping. Set $T(t)x:=C^{-1}S_{0}(t)x,$ $t\geq 0,$ $x\in Z_{1}(A)$ and $m_{k}:=f(k),$ $k\in {\mathbb N}.$
Suppose that there are a number $t_{0}>0,$ a dense subset $E_{0}$ of $E$ and mappings $S_{n} : E_{0} \rightarrow R(C)$ ($n\in {\mathbb N}$) such that
the following conditions hold for all $y\in E_{0}$:
\begin{itemize} 
\item[(i)] The series $\sum_{n=1}^{k}T(t_{0}\lfloor m_{k}\rfloor)S_{\lfloor m_{k-n}\rfloor}y$ converges unconditionally, uniformly in $k\in {\mathbb N}.$
\item[(ii)] The series $\sum_{n=1}^{\infty}T(t_{0}\lfloor m_{k}\rfloor)S_{\lfloor m_{k+n} \rfloor}y$ converges unconditionally, uniformly in $k\in {\mathbb N}.$
\item[(iii)] The series $\sum_{n=1}^{\infty}S_{\lfloor m_{n} \rfloor}y$ converges unconditionally, uniformly in $n\in {\mathbb N}.$
\item[(iv)] $\lim_{n\rightarrow \infty}T(t_{0}\lfloor m_{n} \rfloor)S_{\lfloor m_{n} \rfloor}y=y.$
\item[(v)] $R(C)$ is dense in $E.$
\end{itemize}
Then $(S_{0}(t))_{t\geq 0}$ is $f$-frequently hypercyclic and the operator $T(t_{0})$ is l-$(m_{k})$-frequently hypercyclic.
\end{thm}

\begin{proof}
Without loss of generality, we may assume that $t_{0}=1.$ It is clear that $(m_{k})$ is an increasing sequence in $[1,\infty).$ Define the sequence of operators $(T_{n})_{n\in {\mathbb N}}\subseteq L([R(C)], E)$
by $T_{n}x:=T(n)x,$ $n\in {\mathbb N},$ $x\in R(C).$ Due to \eqref{C-DS}, we get that $T_{n}x=T(1)^{n}x$ for $x\in R(C).$ Then the prescribed assumptions (i)-(iv) in combination with \cite[Theorem 3.1]{1211212018} imply that 
the sequence $(T_{n})_{n\in {\mathbb N}}$
is l-$(m_{k})$-frequently hypercyclic, which means that there exists an element
$x=Cy\in R(C),$ for some $y\in E,$ satisfying that for each open non-empty subset $V'$ of $E$ there exists an increasing sequence $(k_{n})$ of positive integers 
such that the interval $[1,f(k_{n})]$ contains at least $k_{n}c$ elements of set $\{ k\in {\mathbb N} : T_{k}Cy=S_{0}(k)y \in V'\}.$ Since 
$T_{n}\subseteq T(1)^{n},$ the above clearly implies that the operator $T(1)$ is l-$(m_{k})$-frequently hypercyclic with $x=Cy$ being its  l-$(m_{k})$-frequently hypercyclic vector.
We will prove that $Cx=C^{2}y$ is an $f$-frequently hypercyclic vector 
for $(S_{0}(t))_{t\geq 0}$, i.e., that for each open non-empty subset $V$ of $E$ we have $\underline{d}_{f}(\{ t\geq 0 :T(t)Cx=S_{0}(t)Cy \in V \bigr\})>0;$ see also the proof of \cite[Proposition 2.1]{man-peris}. Let such a set $V$ be given. Then, due to our assumption (v), there exist an element $z\in E$ and a positive integer $n\in {\mathbb N}$ such that $L_{n}(Cz,\epsilon) \subseteq V.$ By the local equicontinuity of $(S_{0}(t))_{t\geq 0}$ and the continuity of $C$, we get that there exist an integer $m\in {\mathbb N}$ and a positive constant $c>1$ such that $p_{n}(Cx)\leq c p_{m}(x),$ $x\in E$ and
\begin{align}
\notag p_{n}\bigl( & S _{0}(k+\delta)Cy-S_{0}(k)Cy\bigr) 
\\\notag &  \leq p_{n}\bigl(  S _{0}(k+\delta)Cy-Cz\bigr) +p_{n}\bigl(  S _{0}(k)Cy-Cz\bigr) 
\\\notag & \leq p_{n}\bigl(  S_{0}(\delta) \bigl[S _{0}(k)y-z\bigr]\bigr) 
+p_{n}\bigl( S_{0}(\delta)z-Cz \bigr)+p_{n}\bigl(  S _{0}(k)Cy-Cz\bigr) 
\\\label{daci-mile} & \leq cp_{m}\bigl(S_{0}(k)y-z\bigr) +p_{n}\bigl( S_{0}(\delta)z-Cz\bigr),\quad \delta \in [0,1].
\end{align}
Set $V':=L_{m}(z,\epsilon/3c).$ Then, by the foregoing, we know that there exists an increasing sequence $(k_{n})$ of positive integers 
such that the interval $[1,f(k_{n})]$ contains at least $k_{n}c$ elements of set ${\mathrm A}:=\{ k\in {\mathbb N} : T_{k}Cy=S_{0}(k)y \in V'\}.$ For any $k\in {\mathrm A},$ we have $p_{m}(S_{0}(k)y-z)<\epsilon/3c;$ further on, \eqref{daci-mile} yields that there exists a positive constant $\delta_{0}>0$ such that 
$p_{n}(  S _{0}(k+\delta)Cy-S_{0}(k)Cy) < 2\epsilon/3$ for all $\delta \in [0,\delta_{0}].$ This implies
\begin{align*}
p_{n}\bigl( & S _{0}(k+\delta)Cy-Cz\bigr)
\\ & \leq p_{n}\bigl(  S _{0}(k+\delta)Cy-S_{0}(k)Cy\bigr)+p_{n}\bigl( S_{0}(k)Cy-Cz\bigr)
<2\epsilon/3 +\epsilon/3=\epsilon ,
\end{align*} 
for any $k\in {\mathrm A}$ and $\delta \in [0,\delta_{0}].$
By virtue of this, we conclude that $S _{0}(k+\delta)Cy\in V$ for any $k\in {\mathrm A}$ and $\delta \in [0,\delta_{0}],$ finishing the proof of theorem in a routine manner.
\end{proof}

Plugging $f(t):=t^{q}+1,$ $t\geq 0$ ($q\geq 1$), we obtain a sufficient condition for $q$-frequent hypercyclicity of $(S_{0}(t))_{t\geq 0}$  and $T(t_{0}).$

\begin{rem}\label{prcko-tres}
Consider the situation of Theorem \ref{oma} with $A$ being a subgenerator of a global $\alpha$-times integrated $C$-semigroup $(S_{\alpha}(t))_{t\geq 0}$ on $X,$ $n \geq \lceil \alpha \rceil$  and the Fr\'echet space $[C(D(A^{n}))]$ being separable. 
It is well known that
$C(D(A^{n}))\subseteq Z_{1}(A);$ if $x=Cy\in C(D(A^{n})),$ then for every $t\geq 0,$
\begin{align*}
G\bigl( \delta_{t}  \bigr)x=\frac{d^{n}}{dt^{n}}C^{-1}S_{n}(t)x=\frac{d^{n}}{dt^{n}}S_{n}(t)y
= S_{n}(t)A^{n}y+\sum \limits_{i=0}^{n-1}\frac{t^{n-i-1}}{(n-i-1)!}CA^{n-1-i}y.
\end{align*}
Furthermore, for every $t\geq 0,$ the mapping $G(\delta_{t}) : [C(D(A^{n}))] \rightarrow X$ is linear and continuous as well as the operator family $(G(\delta_{t}))_{t\geq 0}\subseteq L([C(D(A^{n}))] , X)$ is strongly continuous; see e.g.
the proof of \cite[Theorem 5.4]{266}. But, it is not clear how to prove that the  l-$(m_{k})$-frequent hypercyclicity of a single operator $G(\delta_{t_{0}})_{| C(D(A^{n}))},$ for some $t_{0}>0,$ implies the l-$(m_{k})$-frequent hypercyclicity of $(S_{n}(t))_{t\geq 0}$ because an analogue of the estimate \eqref{daci-mile} seems to be not attainable for integrated $C$-semigroups. The interested reader may try to prove an analogue of \cite[Theorem 3.1.32]{knjigah} for l-$(m_{k})$-frequent hypercyclicity. 
\end{rem}

Suppose now that ${\mathcal F}\in P(P({\mathbb N}))$ and ${\mathcal F}\neq \emptyset.$ If ${\mathcal F}$ satisfies the following property:
\begin{itemize}
\item[(I)] $A\in{\mathcal F}$ and $A\subseteq B$ imply $B\in{\mathcal F},$ 
\end{itemize}
then it is said that 
${\mathcal F}$ is a Furstenberg family; a proper Furstenberg family ${\mathcal F}$ is any Furstenberg family satisfying that $\emptyset \notin {\mathcal F}.$ See \cite{furstenberg} for more details.

From the proof of Theorem \ref{oma}, we may deduce the following:

\begin{prop}\label{sot}
Let ${\mathcal F}$ be a Furstenberg family.
Suppose that $A$ is a subgenerator of a global $C$-regularized semigroup $(S_{0}(t))_{t\geq 0}$ on $E$ and $T(t)x:=C^{-1}S_{0}(t)x,$ $t\geq 0,$ $x\in Z_{1}(A).$ 
If $R(C)$ is dense in $E,$ $t_{0}>0$ and $x\in  Z_{1}(A)$ is an ${\mathcal F}$-hypercyclic element of $T(t_{0}),$ then $x$ is an ${\mathcal F}'$-hypercyclic element of $(S_{0}(t))_{t\geq 0},$ where
\begin{align}\label{profica}
{\mathcal F}'=\Biggl\{ B\subseteq [0,\infty) : (\exists A\in {\mathcal F})\, (\exists \delta_{0}>0)\, \bigcup_{k\in A}[k,k+\delta_{0}] \subseteq B \Biggr\}.
\end{align}
\end{prop}

An upper Furstenberg family is any proper Furstenberg family ${\mathcal F}$ satisfying the following two conditions:
\begin{itemize}
\item[(II)] There are a set $D$ and a countable set $M$ such that ${\mathcal F}=\bigcup_{\delta \in D} \bigcap_{\nu \in M}{\mathcal F}_{\delta,\nu},$ where for each $\delta \in D$ and $\nu \in M$ the following holds: If $A\in {\mathcal F}_{\delta,\nu},$ then there exists a
finite subset $F\subseteq {\mathbb N}$ such that the implication $A\cap F \subseteq B \Rightarrow B\in {\mathcal F}_{\delta,\nu}$ holds true.
\item[(III)] If $A\in {\mathcal F},$ then there exists $\delta \in D$ such that, for every $n\in {\mathbb N},$ we have $A-n\equiv \{k-n: k\in A,\ k>n\}\in {\mathcal F}_{\delta},$ where ${\mathcal F}_{\delta}\equiv \bigcap_{\nu \in M}{\mathcal F}_{\delta,\nu}.$
\end{itemize} 

Appealing to \cite[Theorem 22]{boni-upper} in place of \cite[Theorem 3.1]{1211212018}, and repeating almost literally the arguments given in the proof of Theorem \ref{oma}, we may deduce the following result:

\begin{thm}\label{oma-duo}
Suppose that ${\mathcal F}=\bigcup_{\delta \in D} \bigcap_{\nu \in M}{\mathcal F}_{\delta,\nu}$ is an upper Furstenberg family and $A$ is a subgenerator of a global $C$-regularized semigroup $(S_{0}(t))_{t\geq 0}$ on $E.$ Set $T(t)x:=C^{-1}S_{0}(t)x,$ $t\geq 0,$ $x\in Z_{1}(A).$
Suppose that there are a number $t_{0}>0,$ two dense subsets $E_{0}'$ and $E_{0}''$ of $E$ and mappings $S_{n} : E_{0}'' \rightarrow R(C)$ ($n\in {\mathbb N}$) such that
for any $y\in E_{0}''$ and $\epsilon >0$ there exist $A\in{\mathcal F}$ and $\delta \in D$ such that:
\begin{itemize} 
\item[(i)] For every $x\in E_{0}',$ there exists some $B\in {\mathcal F}_{\delta},$ $B\subseteq A$ such that, for every $n\in B,$ one has $S_{0}(t_{0}n)x\in L(0,\epsilon).$
\item[(ii)] The series $\sum_{n \in A}S_{n}y$ converges.
\item[(iii)] For every $m\in A,$ we have $T(mt_{0})\sum_{n \in A}S_{n}y-y\in L(0,\epsilon).$
\item[(iv)] $R(C)$ is dense in $E.$
\end{itemize}
Then the operator $T(t_{0})$ is ${\mathcal F}$-hypercyclic and $(S_{0}(t))_{t\geq 0}$ is ${\mathcal F}'$-hypercyclic, where ${\mathcal F}'$ is given by \eqref{profica}.  
\end{thm}

\begin{rem}\label{special}
Collection of all non-empty subsets $A\subseteq [0,\infty)$ for which $\overline{d}_{qc}(A)>0$ forms an upper Furstenberg family (\cite{boni-upper}, \cite{1211212018}), so that Theorem \ref{oma-duo} with $f(t)=t^{q}+1,$ $t\geq 0$ ($q\geq 1$) gives a sufficient condition for the upper $q$-frequent hypercyclicity of $(S_{0}(t))_{t\geq 0}$  and $T(t_{0}).$
It can be simply proved that the validity of condition $\limsup_{t\rightarrow \infty}\frac{f(t)}{t}>0$
for an increasing function $f  : [0,\infty) \rightarrow [1,\infty)$ implies that the collection of all non-empty subsets $A\subseteq [0,\infty)$ such that $\overline{d}_{f}(A)>0$ forms an upper Furstenberg family, as well.
\end{rem}

We continue by stating two intriguing consequences of 
Theorem \ref{oma}. The first one is motivated by the well-known results of S. El Mourchid \cite[Theorem 2.1]{samir} and E. M. Mangino, A. Peris \cite[Corollary 2.3]{man-peris}; see also \cite[Theorem 3.1.40]{knjigah}.

\begin{thm}\label{3.1.4.13}
Let $t_0>0,$ let ${\mathbb K}={\mathbb C},$ and let $A$ be a subgenerator of a global $C$-regularized semigroup $(S_{0}(t))_{t\geq 0}$ on $E.$ Suppose that $R(C)$ is dense in $E.$ Set $T(t)x:=C^{-1}S_{0}(t)x,$ $t\geq 0,$ $x\in Z_{1}(A).$

\emph{(i)} 
Assume that there exists a family $(f_{j})_{j\in \Gamma}$ of locally bounded measurable mappings $f_{j} : I_{j} \rightarrow E$ such that $I_{j}$ is an interval in ${\mathbb R}$, $Af_{j}(t) = itf_{j}(t)$ for every
$t \in I_{j} ,$ $ j \in \Gamma$ and span$\{f_{j}(t) : j \in \Gamma,\ t \in I_{j}\}$ is dense in $E.$
If $f_{j} \in C^{2}(I_{j} : X)$ for every $j \in \Gamma,$
then $(S_{0}(t))_{t\geq 0}$ is frequently hypercyclic
and each single operator $T(t_0)$ is frequently hypercyclic.
\smallskip

\emph{(ii)} Assume that there exists a family $(f_{j})_{j\in \Gamma}$ of twice continuously differentiable mappings $f_{j} : I_{j} \rightarrow E$ such that $I_{j}$ is an interval in ${\mathbb R}$ and $Af_{j}(t) = itf_{j}(t)$ for every
$t \in I_{j} ,$ $ j \in \Gamma .$ 
Set $\tilde{E}:=\overline{span\{f_{j}(t) : j \in \Gamma,\ t \in I_{j}\}}$.
Then $A_{|\tilde{E}}$ is a subgenerator of a global $C_{|\tilde{E}}$-regularized semigroup $(S_{0}(t)_{| \tilde{E}})_{t\geq 0}$ on $\tilde{E},$
$(S_{0}(t)_{|\tilde{E}})_{t\geq 0}$ is frequently hypercyclic in $\tilde{E}$ and the operator $T(t_{0})_{|\tilde{E}}$
is frequently hypercyclic in $\tilde{E}.$
\end{thm}

\begin{proof}
Consider first the statement (i). Arguing as in the Banach space case \cite[Corollary 2.3]{man-peris}, we get there exists a family 
$(g_{j})_{j\in \Lambda}$ of functions $g_{j}\in  C^{2}({\mathbb R} : E)$ with
compact support such that $Ag_{j}(t) = itg_{j}(t)$ for every $t\in {\mathbb R},$ $j\in \Lambda$ and
$span\{g_{j}(t) : j\in \Lambda,\ t\in {\mathbb R}\}$ is dense in $E.$ 
 For every $\lambda \in \Lambda$ and $r\in {\mathbb R},$ set $\psi_{r,\lambda}:=\int_{-\infty}^{\infty}e^{-irs}g_{\lambda}( s)\, ds.$
Then we have 
\begin{align}\label{jednazba}
T(t)\psi_{r,\lambda}=\psi_{r-t,\lambda},\quad t\geq 0,\ r\in {\mathbb R},\ \lambda \in \Lambda
\end{align} 
and the part (i) follows by applying Theorem \ref{oma} with the sequence $m_{k}:=k$ ($k\in {\mathbb N}$),
$E_{0}:=C(span\{g_{j}(t) : j\in \Lambda,\ t\in {\mathbb R}\})$ and the operator $S_{n}: E_{0}\rightarrow R(C)$ given by 
$S_{n}(C\psi_{r,\lambda}):=C\psi_{t_{0}n+r,\lambda}$ ($n\in {\mathbb N},$ $r\in {\mathbb R},$ $\lambda \in \Lambda$) and after that linearly extended to $E_{0}$ in the obvious way; here, it is only worth noting that the conditions (i)-(iii) follow from \eqref{jednazba} and the fact that the series $\sum_{n=1}^{\infty}\psi_{t_{0}n+r,\lambda}$ and $\sum_{n=1}^{\infty}\psi_{-t_{0}n+r,\lambda}$ converge absolutely (and therefore, unconditionally) since for each seminorm $p_{l}(\cdot),$ where $l\in {\mathbb N},$ 
there exists a finite constant $c_{l}>0$ such that $p_{l}(\psi_{t_{0}n+r,\lambda})+p_{l}(\psi_{-t_{0}n+r,\lambda}) \leq c_{l}n^{-2},$ $n\in {\mathbb N}$ ($r\in {\mathbb R},$ $\lambda \in \Lambda$). This can be seen by applying integration by parts twice, as in the proof of \cite[Lemma 9.23(b)]{Grosse}.
For the proof of (ii), it is enough to observe that an elementary argumentation shows that $A_{|\tilde{E}}$ is a subgenerator of a global $C_{|\tilde{E}}$-regularized semigroup $(S_{0}(t)_{| \tilde{E}})_{t\geq 0}$ on $\tilde{E}.$ Then we can apply (i) to finish the proof.
\end{proof}

The following application of Theorem \ref{3.1.4.13} is quite illustrative ($C=I$):

\begin{example}\label{apa}
Consider the operator $A:=d/dt,$ acting with maximal domain in the Banach space $E:=BUC({\mathbb R}),$ consisting of all bounded uniformly continuous functions.  Then $\sigma_{p}(A)=i{\mathbb R}$ and
$Ae^{\lambda \cdot}=\lambda e^{\lambda \cdot},$ $\lambda \in i{\mathbb R}.$ It is well-known that the space $\tilde{E}:=\overline{span\{e^{\lambda \cdot} : \lambda \in i{\mathbb R}\}}$ coincide with the space of all almost-periodic functions $AP({\mathbb R});$ see \cite{diagana} and \cite{gaston} for more details on the subject.
Due to Theorem \ref{3.1.4.13}(ii), we have that the translation semigroup $(T(t))_{t\geq 0}$ is frequently hypercyclic in $AP({\mathbb R} )$ and, for every $t>0$, the operator $T(t)$ is frequently hypercyclic in $AP({\mathbb R})$; the same holds if frequent hypercyclicity is replaced with Devaney chaoticity or topologically mixing property (\cite{knjigah}). We can similarly prove that the translation semigroup is frequently hypercyclic in the Fr\'echet space $C({\mathbb R} )$ and that, for every $t>0$, the translation operator $f\mapsto f(\cdot +t),$ $f\in C({\mathbb R})$ is frequently hypercyclic in $C({\mathbb R})$.
\end{example}

The subsequent version of Desch-Schappacher-Webb criterion for frequent hypercyclicity can be proved similarly; it is, actually, a simple consequence of Theorem \ref{3.1.4.13} (see also \cite[Theorem 3.1.36]{knjigah}). 

\begin{thm}\label{3018} 
Let $t_0>0,$ let ${\mathbb K}={\mathbb C},$ and let $A$ be a subgenerator of a global $C$-regularized semigroup $(S_{0}(t))_{t\geq 0}$ on $E.$ Suppose that $R(C)$ is dense in $E.$ Set $T(t)x:=C^{-1}S_{0}(t)x,$ $t\geq 0,$ $x\in Z_{1}(A).$

\emph{(i)} 
Assume that there exists an open connected subset $\Omega$ of $\mathbb{C}$,
which satisfies $\sigma_p(A)\supseteq\Omega$ and intersects the imaginary axis,
and $f:\Omega\to E$ is an analytic mapping satisfying $f(\lambda)\in N (A-\lambda)\setminus\{0\}$,
$\lambda\in\Omega$.
Assume, further, that $(x^*\circ f)(\lambda)=0$, $\lambda\in\Omega$,
for some $x^*\in E^*$, implies $x^*=0$.
Then $(S_{0}(t))_{t\geq 0}$ is frequently hypercyclic
and each single operator $T(t_0)$ is frequently hypercyclic.

\smallskip
\emph{(ii)}
Assume that there exists an open connected subset $\Omega$ of $\mathbb{C}$,
which satisfies $\sigma_p(A)\supseteq\Omega$ and intersects the imaginary axis,
and $f:\Omega\to E$ is an analytic mapping satisfying $f(\lambda)\in N(A-\lambda)\setminus\{0\}$,
$\lambda\in\Omega$.
Put $E_0:=span\{f(\lambda):\lambda\in\Omega\}$ and $\tilde{E}:=\overline{E_0}$.
Then $A_{|\tilde{E}}$ is a subgenerator of a global $C_{|\tilde{E}}$-regularized semigroup $(S_{0}(t)_{| \tilde{E}})_{t\geq 0}$ on $\tilde{E},$
$(S_{0}(t)_{|\tilde{E}})_{t\geq 0}$ is frequently hypercyclic in $\tilde{E}$ and the operator $T(t_{0})_{|\tilde{E}}$
is frequently hypercyclic in $\tilde{E}.$
\end{thm}

Using Theorem \ref{3018} and the proof of \cite[Theorem 3.1.38]{knjigah} (see also \cite[Theorem 2.2.10]{knjigaho}), we may deduce the following result:

\begin{thm}\label{kakosteva}
Let $\theta\in(0,\frac{\pi}{2}),$ let ${\mathbb K}={\mathbb C},$ and let $-A$ generate an exponentially equicontinuous, analytic strongly continuous semigroup of angle $\theta$.
Assume $n\in\mathbb{N}$, $a_n>0$, $a_{n-i}\in\mathbb{C}$, $1\leq i\leq n$, $D(p(A))=D(A^n)$,
$p(A)=\sum_{i=0}^na_iA^i$ and $n(\frac{\pi}{2}-\theta)<\frac{\pi}{2}$.

\smallskip
\emph{(i)} Suppose there exists an open connected subset $\Omega$ of $\mathbb{C}$,
satisfying $\sigma_p(-A)\supseteq\Omega$, $p(-\Omega)\cap i\mathbb{R}\neq\emptyset$,
and $f:\Omega\to E$ is an analytic mapping
satisfying $f(\lambda)\in N(-A-\lambda)\setminus\{0\}$, $\lambda\!\in\!\Omega$.
Let $(x^*\!\circ\!f)(\lambda)=0$, $\lambda\in\Omega$,
for some $x^*\in E^*$ imply $x^*=0$.
Then, for every $\alpha\in(1,\frac{\pi}{n\pi-2n\theta})$, there exists $\omega\in\mathbb{R}$
such that $p(A)$ generates an entire $e^{-(p(A)-\omega)^{\alpha}}$-regularized group $(S_{0}(t))_{t\in\mathbb{C}}$.
Furthermore, $(S_{0}(t))_{t\geq 0}$ is frequently hypercyclic and,
for every $t>0$, the operator $C^{-1}S_{0}(t)$ is frequently hypercyclic.

\smallskip
\emph{(ii)} Suppose there exists an open connected subset $\Omega$ of $\mathbb{C}$,
satisfying $\sigma_p(-A)\supseteq\Omega$, $p(-\Omega)\cap\,i\mathbb{R}\neq\emptyset$,
and $f:\Omega\to E$ is an analytic mapping
satisfying $f(\lambda)\in N(-A-\lambda)\setminus\{0\}$, $\lambda\in\Omega $.
Let $E_0$ and $\tilde{E}$ be as in the formulation of Theorem~\emph{\ref{3018}(ii)}.
Then there exists $\omega\in\mathbb{R}$ such that,
for every $\alpha\in(1,\frac{\pi}{n\pi-2n\theta})$, $p(A)$
generates an entire $e^{-(p(A)-\omega)^{\alpha}}$-regularized group $(S_{0}(t))_{t\in\mathbb{C}}$
such that $(S_{0}(t)_{|\tilde{E}})_{t\geq 0}$ is frequently hypercyclic and,
for every $t>0$, the operator $C^{-1}S_{0}(t)_{\tilde{E}}$ is frequently hypercyclic.
\end{thm}

Theorem \ref{3.1.4.13}, Theorem \ref{3018} and Theorem \ref{kakosteva} can be applied in a great number of concrete situations. In  what follows, we will continue our analyses from \cite[Example 3.1.40, Example 3.1.41, Example 3.1.44]{knjigah}:

\begin{example}\label{freja}
\begin{itemize}
\item[(i)] (\cite{fund}) Consider the following convection-diffusion type equation of the form
\[\left\{\begin{array}{l}
u_t=au_{xx}+bu_x+cu:=-Au,\\[0.1cm]
u(0,t)=0,\;t\geq 0,\\[0.1cm]
u(x,0)=u_0(x),\;x\geq 0.
\end{array}\right.
\]
As it is well known, the operator $-A$, acting with domain $D(-A)=\{f\in W^{2,2}([0,\infty)):f(0)=0\}$,
generates an analytic strongly continuous semigroup of angle $\pi/2$ in the space $E=L^2([0,\infty))$,
provided $a$, $b,$ $c>0$ and $c<\frac{b^2}{2a}<1$.
The same conclusion holds true if we consider the operator $-A$
with the domain $D(-A)=\{f\in W^{2,1}([0,\infty)):f(0)=0\}$ in $E=L^1([0,\infty))$.
Set
$$
\Omega :=\Biggl\{\lambda\in\mathbb{C}:\Bigl|\lambda-\Bigl(c-\frac{b^2}{4a}\Bigr)\Bigr|\leq
\frac{b^2}{4a},\;\Im\lambda\neq 0\text{ if }\Re\lambda\leq c-\frac{b^2}{4a}\Biggr\}.
$$
Let $p(x)=\sum_{i=0}^na_ix^i$ be a nonconstant polynomial such that $a_n>0$
and $p(-\Omega)\cap i\mathbb{R}\neq\emptyset$ (this condition holds provided that $a_0\in i\mathbb{R}$).
An application of Theorem \ref{kakosteva}(i) shows that there exists an injective operator $C\in L(E)$ such that $p(A)$
generates an entire $C$-regularized group $(S_{0}(t))_{t\geq 0}$
satisfying that $(S_{0}(t))_{t\geq 0}$ is frequently hypercyclic and each single operator $T(t_{0})$ is frequently hypercyclic ($t_{0}>0$). 
\item[(ii)] (\cite{ji}) Let $X$ be a symmetric space of non-compact type (of rank
one) and $p>2.$ Then there exists an injective operator $C\in L(L^{p}_{\natural}(X))$
such that for each $c\in {\mathbb R}$ the operator
$\Delta_{X,p}^{\natural}-c$ generates an entire
$C$-regularized group $(S_{0}(t))_{t\geq 0}$ in
$L^{p}_{\natural}(X).$ Furthermore, owing to \cite[Theorem 3.1]{ji} and Theorem \ref{kakosteva}(i), there exists a number $c_{p}>0$ such that, for every $c>c_{p},$ the semigroup
$(S_{0}(t))_{t\geq 0}$ is frequently hypercyclic in
$L^{p}_{\natural}(X)$ 
and each single operator $T(t_{0})$ is frequently hypercyclic in
$L^{p}_{\natural}(X)$  ($t_{0}>0$). 
\item[(iii)] (\cite{transfer}, \cite{knjigaho}) Suppose that $\alpha>0$, $\tau\in i\mathbb{R}\setminus\{0\}$ and $E:=BUC(\mathbb{R})$.
After the usual matrix conversion to a first order system, the equation $\tau u_{tt}+u_t=\alpha u_{xx}$ becomes
\[
\frac{d}{dt}\vec{u}(t)=P(D)\vec{u}(t),\;t\geq 0,
\text{ where }D\equiv-i\frac{d}{dx},\;
P(x)\equiv\begin{bmatrix}0 & 1\\-\frac{\alpha}{\tau}x^2 &-\frac{1}{\tau}\end{bmatrix},
\]
and $P(D)$ acts on $E\oplus E$ with its maximal distributional domain.
The polynomial matrix $P(x)$ is not Petrovskii correct
and applying \cite[Theorem 14.1]{l1} we get that there exists an injective operator $C\in L(E\oplus E)$
such that $P(D)$ generates an entire $C$-regularized group $(S_{0}(t))_{t\geq 0}$, with $R(C)$ dense.
Define the numbers $\omega_{1},\ \omega_{2} \in [0,+\infty]$ and functions $\psi_{r,j}\in E\oplus E$ ($r\in\mathbb{R}$, $j=1,2$) as it has been done in \cite[Example 3.1.44]{knjigah};
$\tilde{E}:=\overline{span\{\psi_{r,j}:r\in\mathbb{R},\;j=1,2\}}$.
Due to Theorem \ref{3.1.4.13}(ii), we have that $(S_{0}(t)_{|\tilde{E}})_{t\geq 0}$ is frequently hypercyclic  in $\tilde{E}$ and, 
for every $t>0$, the operator $C^{-1}S_{0}(t)_{|\tilde{E}}$ is frequently hypercyclic in $\tilde{E}.$
\item[(iv)] (\cite{cycch}) Denote by $(W_{Q}(t))_{t\geq 0}$ the $e^{-(-B^{2})^{N}}$-regularized semigroup generated by the operator $Q(B),$ whose existence has been proved in \cite[Lemma 5.2]{cycch}. If the requirement stated in the formulation of \cite[Theorem 5.3]{cycch} holds, then $(W_{Q}(t))_{t\geq 0}$ and each single operator $e^{(-B^{2})^{N}}W_{Q}(t_{0})$ is frequently hypercyclic ($t_{0}>0$); this simply follows from an application of Theorem \ref{3018}(i).
\item[(v)] (\cite{knjigah}) Finally, we turn our attention to integrated semigroups. Let $n\in {\mathbb N},$
$\rho(t):=\frac{1}{t^{2n}+1},\ t\in {\mathbb R},$ $Af:=f^{\prime},$
$D(A):=\{f\in C_{0,\rho}({\mathbb R}) : f^{\prime} \in
C_{0,\rho}({\mathbb R})\},$ $E_{n}:=(C_{0,\rho}({\mathbb
R}))^{n+1},$ $D(A_{n}):=D(A)^{n+1}$ and $A_{n}(f_{1},\cdot \cdot
\cdot ,f_{n+1}):=(Af_{1}+Af_{2},Af_{2}+Af_{3},\cdot \cdot \cdot ,
Af_{n}+Af_{n+1},Af_{n+1}),$ $(f_{1},\cdot \cdot \cdot, f_{n+1}) \in
D(A_{n}).$ Then $\pm A_{n}$
generate global polynomially bounded $n$-times integrated semigroups
$(S_{n,\pm}(t))_{t\geq 0}$ and neither $A_{n}$ nor $-A_{n}$
generates a local $(n-1)$-times integrated semigroup. If we denote by
$G_{\pm,n}$ the associated distribution semigroups generated by $\pm A_{n},$ then for every $\varphi_{1},\cdot\cdot \cdot,
\varphi_{n+1} \in {\mathcal D},$ we have:
$$
G_{\pm,n}(\delta_{t})\bigl(\varphi_{1},\cdot\cdot \cdot ,
\varphi_{n+1}\bigr)^{T}=\bigl(\psi_{1},\cdot\cdot \cdot , \psi_{n+1}\bigr)^{T},
$$
where 
$$
\psi_{i}(\cdot)=\sum \limits_{j=0}^{n+1-i}\frac{(\pm
t)^{j}}{j!}\varphi_{i+j}^{(j)}(\cdot \pm t),\quad 1\leq i\leq n+1.
$$
Set $E_{0}:={\mathcal D}^{n+1}$ and 
$
S_{k}(\varphi_{1}, \cdot\cdot \cdot ,
\varphi_{n+1})^{T}:=(\phi_{1},\cdot\cdot \cdot , \phi_{n+1})^{T},
$
where 
$$
\phi_{i}(\cdot)=\sum \limits_{j=0}^{n+1-i}\frac{(\mp
kt_{0})^{j}}{j!}\varphi_{i+j}^{(j)}(\cdot \mp kt_{0}),\quad 1\leq i\leq n+1,
$$
for any $k\in {\mathbb N},$ $t_{0}>0$ and $\varphi_{1}, \cdot\cdot \cdot ,
\varphi_{n+1} \in {\mathcal D}.$ Then we can simply verify (see also \cite[Example 3.2.39]{knjigah}) that the conditions of Theorem \ref{oma} hold with $C=(\lambda \mp A_{n})^{-n},$ where $\rho(\pm A_{n}) \ni \lambda>0$ is sufficiently large, since the series in (i)-(iii) from the formulation of this theorem converge absolutely. Hence, the integrated semigroups $(S_{n,\pm}(t))_{t\geq 0}$ are frequently hypercyclic in
$E_{n}$ and for each each number $t_{0}>0$ the single operators
$G_{\pm,n}(\delta_{t_{0}})$ are frequently hypercyclic in
$E_{n}$.
\end{itemize}
\end{example}

We close the paper with the observation that, for any $C$-regularized semigroup or integrated semigroup considered above, say $(S(t))_{t\geq 0},$ any finite direct sum $(S(t)\oplus S(t)\oplus \cdot\cdot \cdot \oplus S(t))_{t\geq 0}$ is again frequently hypercyclic or subspace frequently hypercyclic, with the meaning clear. The same holds for finite direct sums of considered single operators (cf. \cite{qjua} and \cite{kimpark} for more details about this topic).

\end{document}